\documentclass[11pt]{article}

\usepackage{amsmath,amsthm}

\usepackage{amssymb,latexsym}

\usepackage[mathscr]{euscript}

\usepackage{enumerate,color}

\usepackage{enumerate}

\newcommand{\BB}{{\cal B}}

\newcommand{\EE}{{\cal E}}
\newcommand{\FF}{{\cal F}}

\newcommand{\MM}{{\cal M}}

\newcommand{\BR}{{\mathbb R}}

\newtheorem{theorem}{\bf Theorem}[section]
\newtheorem{proposition}[theorem]{\bf Proposition}
\newtheorem{lemma}[theorem]{\bf Lemma}

\theoremstyle{definition}
\newtheorem{definition}[theorem]{Definition}

\newtheorem{remark}[theorem]{Remark}

\numberwithin{equation}{section}

\begin{document}

\title {Uniqueness for an obstacle problem arising from  logistic-type equations
with fractional Laplacian}
\author {Tomasz Klimsiak\\
{\small Faculty of Mathematics and Computer Science,
Nicolaus Copernicus University} \\
{\small  Chopina 12/18, 87--100 Toru\'n, Poland}\\
{\small E-mail address: tomas@mat.umk.pl}}
\date{}
\maketitle
\begin{abstract}
We prove a uniqueness theorem for  the obstacle problem for   linear equations involving the fractional Laplacian  with  zero  Dirichlet exterior  condition. 
The problem under consideration arises as the limit of some logistic-type equations. 
Our result extends (and slightly strengthens)
the known corresponding results for the classical Laplace operator with zero boundary condition. 
Our proof, as compared with the known proof for the classical Laplace operator, 
is entirely new, and is based on the probabilistic potential theory. 
Its advantage is that it may be applied to
a wide class of integro-differential operators. 
\end{abstract}




\section{Introduction}

Let $D\subset\BR^d$, $d\ge2$, be a bounded Lipschitz domain, $D_0\subset D$ be a bounded {\em Dirichlet regular} domain (with respect to the fractional Laplacian). For $\alpha\in(0,2)$, we denote by  $(\Delta^{\alpha/2})_{|D}$ the  Dirichlet fractional Laplacian   on $D$ with zero exterior condition (see Section \ref{sec2} for details). In case $\alpha=2$, by  $(\Delta^{\alpha/2})_{|D}$ we mean the classical Laplace operator $\Delta_D$ on $D$ with zero boundary condition.
In the present paper, we prove a uniqueness result for the following
obstacle problem:
\begin{equation}
\label{eq1.1} \left\{
\begin{array}{l}\max\big\{-(\Delta^{\alpha/2})_{|D} u-au,u- \mathbb I_{D\setminus \overline D_0}\big\}=0,\medskip\\
u>0,\quad \mbox{on}\quad D,
\end{array}
\right.
\end{equation}
where $a$ is a positive constant, and
\begin{equation}
\label{eq1.1cbgsh} 
\mathbb I_{D\setminus \overline D_0}(x)=\left\{
\begin{array}{l}1,\quad\quad\,\,\quad x\in D\setminus \overline D_0,
\medskip\\ +\infty,\quad\quad x\in \overline D_0.
\end{array}
\right.
\end{equation}
Problem of this type arises in the study of asymptotic behaviour, as $p\rightarrow \infty$,
of logistic type equations
(see \cite{DDM,RT} for the case $\alpha=2$ and \cite{K:arx2} for equations with  $\alpha\in (0,2)$). Problem (\ref{eq1.1}) with $\alpha=2$ also arises as the limit of some  predator-prey models (see \cite{DD1,DHMP}).

From \cite{DDM} (in case $\alpha=2$) and \cite{K:arx2} (in case $\alpha\in(0,2)$) we know  that (\ref{eq1.1}) has a solution if and only if  $a\in[\lambda_1^D,\lambda_1^{D_0})$, where $\lambda_1^D, \lambda_1^{D_0}$ stand for the first eigenvalue of $-(\Delta^{\alpha/2})_{|D}$ and $-(\Delta^{\alpha/2})_{|D_0}$, respectively.
In the present paper, we concentrate on the uniqueness of  solutions to (\ref{eq1.1}). This problem is quite subtle and  difficult.  In \cite{DD2} it is investigated in case $\alpha=2$. Suppose that $D_0, D$ are smooth and $\overline{D}_0\subset D$.
The main result of \cite{DD2}  states that (\ref{eq1.1}) has at most one  solution  if $a\in(\lambda_1^{D},\lambda_1^{D_0})$.
This result is proved by using an equivalent free boundary formulation of (\ref{eq1.1})
and  tools from  the theory of variational inequalities and harmonic functions.

The method used in \cite{DD2} seems to be suitable only for  $\alpha=2$. To deal with
nonlocal operators, we propose a new method. It allows us to prove
that if $a\in(\lambda_1^{D},\lambda_1^{D_0})$, then
for any $\alpha\in(0,2]$ there exists at most one  solution to (\ref{eq1.1}).
This  generalizes the result of \cite{DD2} to nonlocal operators but also slightly strengthens  the known uniqueness result  for $\alpha=2$  because we assume that $D_0\subset D$ and not that $\overline{D}_0\subset D$ as in \cite{DD2}. Moreover, 
we consider less regular than in \cite{DD2} domains $D_0,D$ (see comments in \cite[Remark (i)]{DD2}).

In the present paper, we use a  definition of   solution to (\ref{eq1.1}), which is equivalent to, but  different from that of \cite{DD2}. Let $\EE_D$ denote the Dirichlet form associated with the operator $(\Delta^{\alpha/2})_{|D}$ (see Section \ref{sec2}). By a solution to (\ref{eq1.1}) we mean a strictly positive function $u\in \tilde H^{\alpha/2}(D)$ (closure of $C_c^\infty(D)$ in $H^{\alpha/2}(\BR^d)$) having the property that $u\le \mathbb I_{D\setminus \overline D_0}$ $m$-a.e.  and  such that for some bounded smooth nonnegative Borel measure $\nu$ on $D$ (called the reaction measure for $u$) we have
\[
\EE_D(u,v)=a(u,v)-\int_D\tilde v\,d\nu,\quad v\in \tilde H^{\alpha/2}(D),
\]
where $(\cdot,\cdot)$ stands for the usual inner product in $L^2(D;m)$ and $\tilde v$
is a quasi-continuous $m$-version of $v$ ($m$ stands for the Lebesgue measure on $\BR^d$). Moreover, we require that $u$ satisfies the  minimality condition, which says  that for every quasi-continuous  $\eta$ such that $u\le \eta\le \mathbb I_{D\setminus \overline D_0}$ $m$-a.e. we have
\[
\int_D (\eta-\tilde u)\,d\nu=0,
\]
where $\tilde u$ is a quasi-continuous $m$-version of $u$. In other words, $u$ is a solution of the following equation with measure data

\[
-(\Delta^{\alpha/2})_{|D}u=a u-\nu,
\]
where $\nu$ is a positive measure which acts only when $u$ touches the barrier.
Our proof of uniqueness is based on the following two crucial observations. The first  one is that if $u$ is a solution to (\ref{eq1.1}), then
\[
u\cdot \nu=\nu,
\]
so in fact,
\[
-(\Delta^{\alpha/2})_{|D}u=au-u\cdot \nu.
\]
Equivalently,
\[
\big(-(\Delta^{\alpha/2})_{|D}+\nu\big)u=a u.
\]
This shows that any  solution to (\ref{eq1.1}) is in fact an eigenfunction for the operator $-(\Delta^{\alpha/2})_{|D}+\nu$, i.e. the operator $-(\Delta^{\alpha/2})_{|D}$ perturbed by the smooth bounded measure $\nu$. It is well known that this  operator is a  self-adjoint  nonnegative operator on $L^2(D;m)$ generating a Markov $C_0$-semigroup  of contractions on $L^2(D;m)$. The second  crucial observation is that $\nu$ has  a compact support in $D$. This allows us to  prove, by using  some results of  Hansen \cite{Hansen}, that the Green function $G_D$ for $-(\Delta^{\alpha/2})_{|D}$ is comparable with the Green function $G^\nu_D$ for $-(\Delta^{\alpha/2})_{|D}+\nu$. These two facts, when combined with the sub and  supersolutions method (generalized in the present paper to the case of our obstacle problem), give the uniqueness result.

\section{Potential theory for fractional Laplacian on bounded domain}
\label{sec2}

For fixed $x\in\BR^d$ and $r>0$, we denote $B(x,r)=\{y\in\BR^d: |y-x|<r\}$.
In the whole paper, we assume that $D$ is an open bounded Lipschitz domain in $\BR^d$ with Lipschitz character $(r_0,\lambda)$, $r_0,\lambda>0$,
i.e. open nonempty subset of $\BR^d$ such that: for every $x\in \partial D$ there exists a function $\Gamma_x:\BR^{d-1}\rightarrow \BR$
such that
\[
|\Gamma_x(a)-\Gamma_x(b)|\le\lambda |a-b|,\quad a,b\in\BR^{d-1},
\]  
and there exists an orthogonal coordinate system $S_x$ such that if $y=(y_1,y_2,\dots,y_d)$ in $S_x$, then 
\[
B(x,r_0)\cap D=B(x,r_0)\cap \{y: y_d>\Gamma_x(y_1,y_2,\dots,y_{d-1})\}.
\]
We denote by $m$ the Lebesgue measure on $\BR^d$.
By $\BB(\BR^d)$ (resp. $\BB(D))$, we denote  the set of  Borel measurable functions on $\BR^d$ (on $D$) with 
values in $\BR\cup\{-\infty\}\cup\{\infty\}$. $\BB_b(D)$ is the subset of $\BB(D)$ consisting of all bounded functions.

\subsection{Dirichlet fractional Laplacian }

Let $\alpha\in (0,2)$. To define the Dirichlet fractional Laplacian on $D$, we first set
\[
D(\Delta^{\alpha/2})=\{u\in L^2(\BR^d;m): \int_{\BR^d}|x|^{2\alpha}|\hat u(x)|^2\,m(dx)<\infty\},
\]
and for $u\in D(\Delta^{\alpha/2})$, we set
\[
\widehat {\Delta^{\frac \alpha2} u}(x)=|x|^{\alpha}\hat u(x),\quad x\in\BR^d,
\]
where $\hat u$ stands for the Fourier transform of  $u$.
Let us consider the  form $(\EE, D(\EE))$ defined as
\[
\EE(u,v)=\int_{\BR^d}\hat u(x)\hat v(x)|x|^\alpha\,m(dx),\quad u,v\in D(\EE),
\]
where
\[
D(\EE)=\{u\in L^2(\BR^d;m): \int_{\BR^d}|x|^{\alpha}|\hat u(x)|^2\,m(dx)<\infty\}.
\]
By \cite[Example 1.4.1]{FOT},  $(\EE,D(\EE))$ is a regular symmetric Dirichlet form on $L^2(\BR^d;m)$. In the language of Sobolev spaces, $D(\EE)=H^{\frac \alpha2}(\BR^d)$ (see \cite[page 76]{McLean}).
When $u\in C^\infty_c(\BR^d)$ the more explicit formula for $-\Delta^{\alpha/2}u$ is known (see e.g. \cite[Example 1.4.1]{FOT}):
\[
-\Delta^{\alpha/2}u(x)=c_{d,\alpha}\lim_{r\searrow 0}\int_{\BR^d\setminus B(x,r)}\frac{u(x)-u(y)}{|x-y|^{d+\alpha}}\,dy,\quad x\in \BR^d.
\]

The capacity  Cap$:2^{\BR^d}\rightarrow \BR^+\cup\{\infty\}$ associated with $(\EE,D(\EE))$ is defined as follows: for an arbitrary open set $U\subset\BR^d$, we set
\begin{equation}
\label{eq2.1}
\mbox{\rm Cap}(U)=\inf\{\EE(u,u),\, u\in D(\EE),\, u\ge \mathbf{1}_U\,\, m\mbox{-a.e.}\},
\end{equation}
and then,  for an arbitrary $B\subset\BR^d$, we set
\begin{equation}
\label{eq2.2}
\mbox{\rm Cap}(B)=\inf\{\mbox{\rm Cap}(U): U\supset B,\, U\subset\BR^d,\, U\mbox{  open}\}.
\end{equation}
We say that some property holds $\EE$-quasi everywhere ($\EE$-q.e. in abbreviation) if it holds except possibly   a set of capacity Cap zero. Such sets shall be called $\EE$-exceptional.

Recall that a function $u$ on $\BR^d$ is called $\EE$-quasi-continuous if for every $\varepsilon>0$
there exists a closed set $F_\varepsilon\subset\BR^d$ such that Cap$(\BR^d\setminus F_\varepsilon)\le\varepsilon$
and $u_{|F_\varepsilon}$ is continuous. By \cite[Theorem 2.1.3]{FOT}, each $u\in D(\EE)$ admits
an $\EE$-quasi-continuous $m$-version. In what follows for given function 
$u\in\BB(\BR^d)$, we denote  by $\tilde u$ its   $\EE$-quasi-continuous $m$-version (whenever it exists).

We denote by  $(\EE_D,D(\EE_D))$ the part of $(\EE,D(\EE))$  on $D$. Recall that
\[
D(\EE_D)=\{u\in D(\EE): \tilde u=0,\,\mbox{q.e. on}\, \BR^d\setminus D\},\quad \EE_D(u,v)=\EE(u,v),\quad u,v\in D(\EE_D).
\]
By \cite[Theorem 4.4.3]{FOT}, $(\EE_D,D(\EE_D))$ is a regular symmetric Dirichlet form on $L^2(D;m)$. Therefore, by \cite[Sections 1.3,1.4]{FOT}, there exists a unique self-adjoint nonpositive definite
operator $(A_D,D(A_D))$ on $L^2(D;m)$ such that
\[
D(A_D)\subset D(\EE_D),\qquad \EE_D(u,v)=(-A_Du,v),\quad u\in D(A_D),\, v\in D(\EE_D).
\]
The operator $(A_D,D(A_D))$ is called the Dirichlet fractional Laplacian.
We put
\[
(\Delta^{\alpha/2})_{|D}:= A_D.
\]
By \cite[Theorem 4.4.3]{FOT}, $C_c^\infty(D)$ is a dense subspace of $D(\EE_D)$
in the norm determined by $\EE$. Therefore $D(\EE_D)=\tilde H^{\alpha/2}(D)$ (see \cite[page 77]{McLean}).
For $u\in C_c^\infty(D)$,
\[
(\Delta^{\alpha/2})_{|D}u(x)=\Delta^{\alpha/2}u(x)-\kappa_D(x)u(x),\quad x\in D,
\]
where
\[
\kappa_D(x)=c_{d,\alpha}\int_{\BR^d\setminus D}|x-y|^{-d-\alpha}\,dy,\quad x\in D.
\]
It is worth noting here
that $A_D\neq -(-\Delta_D)^{\alpha/2}$, where $\Delta_D$ is the Laplace operator with zero Dirichlet boundary condition on $D$. The latter operator is called the fractional Dirichlet Laplacian.


By replacing $(\EE,D(\EE))$ by $(\EE_D,D(\EE_D))$ and $\BR^d$ by $D$ in (\ref{eq2.1}) and (\ref{eq2.2}), we define the capacity Cap$_D$ associated with $\EE_D$, and then  we define the notions of $\EE_D$-exceptional sets and $\EE_D$-quasi-continuity. By \cite[Theorem 4.4.4]{FOT}, Cap and Cap$_D$ are strongly equivalent on $D$. Therefore, without ambiguity, we may write q.e.,  exceptional or quasi-continuous instead of $\EE$-q.e., $\EE_D$-q.e. $\EE$-exceptional, $\EE_D$-exceptional or $\EE$-quasi-continuous, $\EE_D$-quasi-continuous.

\subsection{Green functions and  transition functions}

We denote by  $(J_\beta)_{\beta>0}$, $(T_t)_{t\ge 0}$ the resolvent and the $C_0$-semigroup of contractions generated by  $A:=\Delta^{\alpha/2}$, respectively,
and by $(J^D_\beta)_{\beta>0}$, $(T^D_t)_{t\ge 0}$   the resolvent  and the $C_0$-semigroup of contractions generated by $(\Delta^{\alpha/2})_{|D}$, respectively.

Let $\partial$ be a one-point compactification of $D$.  
Let $\cal D$ denote  the set of all 
functions $\omega: [0,\infty)\to \BR^d\cup\{\partial\}$, that are right continuous and
possess the left limits for all $t\ge0$ (c\'adl\'ags), and  have the property that if $\omega(t)=\partial$,
then $\omega(s)=\partial,\, s\ge t$. We equip $\mathcal D$ with
the Skorokhod topology, see e.g. Section 12 of \cite{bil}. 
Define the {\em canonical process}: $X:\mathcal D\to\mathcal D$, $X_t(\omega):=\omega(t)$, $\omega \in \mathcal D$,
{\em shift operator}: $\theta_s:\mathcal D\to\mathcal D$, $\theta_s(\omega)(t)=\omega(t+s)$,  and   {\em life time}: $\zeta: \mathcal D\to [0,\infty]$,   $\zeta(\omega):=\inf\{t>0: X_t(\omega)=\partial\}$.
It is well known that there exists a rotation invariant $\alpha$-stable L\'evy process
$\mathbb{X}=(X,(P_x)_{x\in \BR^d},(\FF_t)_{t\ge0})$ on $\BR^d$
such that for every nonnegative $f\in \BB(\BR^d)$,
\begin{equation}
\label{eq2.4.1}
T_tf(x)=E_x f(X_t):=\int_{\mathcal D}f(X_t(\omega))\,dP_x(\omega),\quad x\in\BR^d,\, t\ge 0,\quad m\mbox{-a.e.}
\end{equation}
We let for $\beta>0$ and $f\in\BB^+(\BR^d)$
\[
R_\beta f(x)=E_x\int_0^\infty e^{-\beta t}f(X_t)\,dt,\quad x\in\BR^d.
\]
In the whole paper 
we adopt the convention that any function $f$ defined on a subset of  $\BR^d$ equals zero at $\partial$.
By \cite[Theorem 4.4.2]{FOT}, there exists a Hunt process $\mathbb{X}^D=(X,(P^D_x)_{x\in D\cup\{\partial\}},(\FF_t)_{t\ge0})$
on $D\cup\{\partial\}$ such that
\begin{equation}
\label{eq2.4.1ab}
T^D_tf(x)=E^D_x f(X_t):=\int_{\mathcal D}f(X_t(\omega))\,dP^D_x(\omega),\quad x\in D,\, t\ge 0,\quad m\mbox{-a.e.}
\end{equation}
Here $(\FF_t)_{t\ge 0}$ is a filtration (non-decreasing sequence of $\sigma$-fields of subsets of $\Omega$).
By \cite[Theorem 4.4.1, Theorem 4.4.2]{FOT}, for all $t\ge 0$ and nonnegative $f\in\BB(D)$,
\begin{equation}
\label{eq2.4.2abc}
P^D_tf(x):=E^D_x f(X_t)=E_x f(X_t)\mathbf{1}_{\{t<\tau_D\}},\quad x\in D,
\end{equation}
where
\[
\tau_D=\inf\{t>0: X_t\in \BR^d\setminus D\}.
\]
Hence 
\begin{equation}
\label{eq2.4.2abcd}
R^D_\beta f(x):=E^D_x\int_0^{\infty}e^{-\beta t}f(X_t)\,dt =E_x\int_0^{\tau_D}e^{-\beta t}f(X_t)\,dt,\quad x\in D.
\end{equation}
It is well known (see \cite[Exercise 4.2.1, Lemma 4.2.4]{FOT}) that there is $G_{\beta}\in\BB(\BR^d)\times \BB(\BR^d)$ (called the $\beta$-Green function) such that for every $f\in \BB^+(\BR^d)$,
\[
R_{\beta}f(x)=\int_{\BR^d}G_{\beta}(x,y)f(y)\,dy\quad  x\in\BR^d.
\]
Similarly, there is  $G_{D,\beta}\in\BB(D)\times \BB(D)$ (called the $\beta$-Green function for  $D$) such that for every $f\in \BB^+(D)$,
\[
R^D_\beta f(x)=\int_DG_{D,\beta}(x,y)f(y)\,dy \quad x\in D.
\]
Given a  nonnegative Borel measure $\mu$ on $D$, we set
\[
R^D_\beta\mu(x)=\int_DG_{D,\beta}(x,y)\,\mu(dy),\quad x\in D.
\]
Note that  $R^D_\beta f=J^D_\beta f$ $m$-a.e. for every $f\in L^2(D;m)$.
Moreover, by \cite[Theorem 4.2.3]{FOT}, $J^D_\beta f\in D(\EE_D)$ and $\widetilde{J^D_\beta f}=R^D_\beta f$ q.e. for $f\in L^2(D;m)$.
By \cite[Exercise 4.2.1]{FOT}, there exists a transition
function $p_D:\mathbb R^+\times D\times D\rightarrow \BR^+$ such that for every $f\in \BB^+(D)$,
\[
P_t^Df(x)=\int_Dp_D(t,x,y)f(y)\,dy\quad x\in D,\, t>0.
\]
It is well known (see e.g. \cite{Grzywny}) that $p_D(t,x,y)$ is finite and strictly positive for $x,y\in D, t>0$.
Given a nonnegative Borel measure $\mu$ on $D$, we set
\[
P^D_t\mu(x)=\int_Dp_D(t,x,y)\,\mu(dy),\quad x\in D,\, t>0.
\]
By \cite[Theorem 4.2.3]{FOT}, $T^D_t f\in D(\EE)$ and $\widetilde{T^D_t f}=P^D_tf$ q.e.  for
all $f\in L^2(D;m)$ and $t>0$. Moreover,
\[
G_{D,\beta}(x,y)=\int_0^\infty e^{-\beta t}p_D(t,x,y)\,dt,\quad x,y\in D.
\]
It follows that we can  set
\[
G_D(x,y):=\sup_{\beta>0}G_{D,\beta}(x,y)=\lim_{\beta\searrow 0}G_{D,\beta}(x,y),\quad x,y\in D.
\]
It is well known that $(\Delta^{\alpha/2})_{|D}: D((\Delta^{\alpha/2})_{|D})\rightarrow L^2(D;m)$ is a bijection. From the above definition of $G_D$ it follows that  for every $f\in L^2(D;m)$,
\[
J^Df:=(-(\Delta^{\alpha/2})_{|D})^{-1}f=\int_DG_D(\cdot,y)\,f(y)\,dy\quad m\mbox{-a.e. on }D.
\]
By \cite[Lemma 2.2.11]{FOT},  $J^Df\in D(\EE_D)$ and $\widetilde{J^Df}=R^Df$ q.e. for every $f\in L^2(D;m)$.
For a nonnegative Borel measure $\mu$ on $D$, we set
\[
R^D\mu(x)=\int_D G_D(x,y)\,\mu(dy),\quad x\in D.
\]
Recall that a positive Borel measurable function $u$ on $D$ is called excessive if 
\[
\sup_{t>0} P^D_t f(x)=f(x),\quad x\in D.
\]
By \cite[Lemma 4.2.4]{FOT}, for every positive Borel measure $\mu$, $R^D\mu$ is an excessive function. 

\subsection{Smooth measures}

An increasing sequence $\{F_n\}$ of closed subsets of $\BR^d$ (resp. $D$) is called a  generalized $\EE$-nest (resp. generalized $\EE_D$-nest)
if for every compact $K\subset \BR^d$ (resp. $K\subset D$), Cap$(K\setminus F_n)\rightarrow 0$ (resp. Cap$_D(K\setminus F_n)\rightarrow0$) as $n\rightarrow \infty$.
A Borel signed measure $\mu$ on $\BR^d$ (resp. $D$) is called $\EE$-smooth (resp. $\EE_D$-smooth)
if it charges no set of capacity Cap (resp. Cap$_D$) zero and there exists a generalized $\EE$-nest (resp. $\EE_D$-nest) such that $|\mu|(F_n)<\infty$, $n\ge 1$, where $|\mu|$ denotes the variation of $\mu$. Note  that if a Borel measure $\mu$ on $D$ is bounded, then it is $\EE$-smooth if and only if  it is $\EE_D$-smooth. This follows from the fact that Cap and Cap$_D$ are equivalent.

We  denote by $\MM_0$ (resp. $\MM_0(D)$) the set of  all $\EE$-smooth (resp. $\EE_D$-smooth) measures on $\BR^d$ (resp. $D$). We also set
\[
\mathbb M_0(D)=\{\mu\in\MM_0(D):R^D|\mu|<\infty \mbox{ q.e.}\},\quad \mathbb \MM_{0,b}(D)=\{\mu\in\MM_0(D):|\mu|(D)<\infty\}.
\]
By \cite[Proposition 5.13]{KR:JFA},
\begin{equation}
\label{eq2.3.1}
\MM_{0,b}(D)\subset \mathbb M_0(D).
\end{equation}
Also note that if $\mu$ is a Borel measure such that $R^D|\mu|<\infty$ on $D$, then
$\mu\in \MM_0(D)$ (see \cite[Exercise 4.2.2]{FOT}). We say that an $\EE_D$-smooth measure
$\mu$ belongs to the class $S_0(D)$ (called the class of measures of finite energy integral) if there exists $c>0$ such that
\[
\int_D|\tilde u|\,d|\mu|\le c\sqrt{\EE_D(u,u)},\quad u\in D(\EE_D).
\]
By \cite[Theorem 2.2.2]{FOT}, for any positive $\mu\in S_0(D)$ and $\eta\in D(\EE_D)$,
\begin{equation}
\label{eq.smfe1}
\EE_D(R^D\mu,\eta)=\int_D\tilde\eta \,d\mu.
\end{equation}

\subsection{Probabilistic potential theory}
\label{sec2.4}

Recall that a bounded subset $D$ of $\BR^d$ is called {\em Dirichlet
regular} (with respect to the fractional Laplacian) if for every $x\in \BR^d\setminus D$,
\[
P_x(\tau_D>0)=0.
\]
By \cite[Example VII.3.4.3]{BH}, each bounded Lipschitz domain is Dirichlet regular. From
this and  \cite{Chung} it follows that $(P^D_t)_{t\ge 0}$ is doubly Feller.
This means that it is strongly Feller, i.e. $P^D_tf\in C_b(D)$ for
every $f\in \mathcal B_b(D)$, and it is Fellerian, i.e. $P^D_t
f\in C_0(D)$ for every $f\in C_0(D)$.

By \cite[Theorem 5.1.4]{FOT}, there is  a one-to-one correspondence between nonnegative $\EE_D$-smooth measures and
positive continuous additive functionals (PCAFs) of $\mathbb X^D$ (see \cite[Section 5.1]{FOT}). By $A^\mu$ we denote the unique PCAF of $\mathbb X^D$ associated with $\mu\in \MM_0(D)$. By \cite[Theorem 5.1.3]{FOT} and (\ref{eq2.4.2abcd}),
\begin{equation}
\label{eq2.4.4}
E^D_x\int_0^{\infty}f(X_r)\,dA^\mu_r=E_x\int_0^{\tau_D}f(X_r)\,dA^\mu_r=\int_DG_D(x,y)f(y)\,\mu(dy)=R^D(f\cdot \mu)(x)
\end{equation}
for q.e. $x\in D$. For a signed measure  $\mu\in \MM_0(D)$ having a decomopsition $\mu=\mu^+-\mu^-$ into a positive and negative part, we  set $A^\mu= A^{\mu^+}-A^{\mu^-}$. Note also that if $\mu\in\MM_{0,b}^+(D)$ and $\sup_{x\in D}R^D\mu<\infty$, then there exists a strict PCAF $A^\mu$ of $\mathbb X^D$ such that (\ref{eq2.4.4}) holds for every $x\in D$ (see \cite[Theorem 5.1.6]{FOT}).

Recall that a c\`adl\`ag process $M$ adapted to $(\FF_t)_{t\ge 0}$ is called a martingale additive functional (MAF) of 
$\mathbb X^D$ iff $M_0=0$, $M$ is an additive functional, $E^D_x|M_t|<\infty$ q.e. $t\ge 0$, and $E^D_xM_t=0$ q.e. $t\ge 0$.
A MAF $M$ of $\mathbb X^D$ is called uniformly integrable iff for q.e. $x\in D$, and every stopping time $\alpha\le\tau_D$,
\[
E^D_x|M_\alpha|<\infty,\qquad E^D_x(M_{\tau_D}|\FF_\alpha)=M_\alpha,\quad P^D_x\mbox{-a.s.}
\] 

The following  very useful result will be frequently used in the sequel.
\begin{lemma}
\label{lm2.4.1}
Assume that  $u\in\BB(D)$ and $\mu\in\mathbb M_0(D)$.
Then
\begin{equation}
\label{eq2.4.5}
u(x)=R^D\mu(x)
\end{equation}
for q.e. $x\in D$ if and only if there exists a uniformly integrable MAF $M$ of $\mathbb X^D$ such that  for q.e. $x\in D$,
\begin{equation}
\label{eq2.4.6}
u(X_t)=\int_t^{\tau_D}\,dA^\mu_r-\int_t^{\tau_D}\,dM_r,\quad t\in [0,\tau_D],\quad P_x\mbox{-a.s.,}
\end{equation}
Moreover, if $\sup_{x\in D}R^D|\mu|(x)<\infty$, then "for q.e. $x\in D$" may be replaced by "for every $x\in D$" with $A^\mu$ being a strict PCAF of $\mathbb X^D$.
\end{lemma}
\begin{proof}
First suppose that (\ref{eq2.4.6}) is satisfied. Taking the
expectation with respect to $P_x$ of both sides of (\ref{eq2.4.6})
with $t=0$, and then using (\ref{eq2.4.4}) with $f=1$ we get
(\ref{eq2.4.5}). Now suppose that  (\ref{eq2.4.5}) holds. Let
$N$
be an exceptional set such that (\ref{eq2.4.5}) holds for $x\in
D\setminus N$.  By
\cite[Theorem 4.1.1]{FOT}, we may assume that $P^D_x(X_t\notin
 N,\, t\ge 0)=1$ for every $x\in D\setminus N$.
Hence, by (\ref{eq2.4.5}), additivity of $A^\mu$, and strong  Markov property, for every $x\in D\setminus N$ and every stopping time $\sigma\le\tau_D$, we
have
\begin{equation}
\label{eq2.4.7}
u(X_\sigma)=E^D_{X_\sigma}\int_0^{\infty}\,dA^\mu_r=E^D_x\Big(\int_\sigma^{\infty}\,dA^\mu_r\big|
\FF_\sigma\Big)=E^D_x\Big(\int_0^{\infty}\,dA^\mu_r\big|
\FF_\sigma\Big)-A^\mu_\sigma \quad
P^D_{x}\mbox{-a.s.}
\end{equation}
Set
\begin{equation}
\label{eq3r.4}
M_t=u(X_t)-u(X_0)+A^\mu_t,\quad t\ge 0.
\end{equation}
By quasi-continuity of $u$
and \cite[Theorem 4.6.1]{FOT}, $M$ is a c\`adl\`ag process.
Clearly, $M$ is an additive functional and $M_0=0$.
By (\ref{eq2.4.7}), 
$M$ is a uniformly integrable martingale on $[0,\tau_D]$ under measure $P^D_x$.
Equivalently, by (\ref{eq2.4.2abc}), $M$ is a uniformly integrable martingale on $[0,\tau_D]$ under measure $P_x$
for $x\in D\setminus N$.
If $\sup_{x\in D}R^D|\mu|<\infty$, then the above argument holds true with $N=\emptyset$ with one exception that instead of \cite[Theorem 4.6.1]{FOT} one need to apply \cite[Theorem III.5.7]{BG}.
\end{proof}

\subsection{Feynman-Kac formula}
\label{sec2.5}

For a nonnegative measure $\mu\in\MM_{0,b}(D)$ such that $\sup_{x\in D}R^D\mu(x)<\infty$, we define  the perturbation of the form $(\EE_D,D(\EE_D))$ by $\mu$ as follows:
\[
D(\EE^\mu_D)=\{u\in D(\EE_D):\int_D\tilde u^2\,d\mu<\infty\},\qquad \EE^\mu_D(u,v)=\EE_D(u,v)+\int_D\tilde u\tilde v\,d\mu.
\]
By \cite[Theorem 6.1.2]{FOT}, $(\EE^\mu_D,D(\EE^\mu_D))$ is a regular symmetric Dirichlet form on $L^2(D;m)$.
Let $(T^{D,\mu}_t)_{t>0}$ be the semigroup generated by $\EE^\mu_D$ and $A^\mu_D$ be its generator. 
We let
\[
-(\Delta^{\alpha/2})_{|D}+\nu:= -A^\nu_D.
\]
By \cite{FOT} there exists a Hunt process $\mathbb X^{D,\mu}=(X,(P^{D,\mu}_x)_{x\in D\cup\{\partial\}},(\FF_t)_{t\ge0})$
on $D\cup\{\partial\}$ such that
\begin{equation}
\label{eq2.4.1ab-fk}
T^{D,\mu}_tf(x)=E^{D,\mu}_x f(X_t):=\int_{\mathcal D}f(X_t(\omega))\,dP^{D,\mu}_x(\omega),\quad x\in D,\, t\ge 0,\quad m\mbox{-a.e.}
\end{equation}
We denote by  $G^\mu_D$ the  Green function for $-(\Delta^{\alpha/2})_{|D}+\nu$ (see \cite[Exercise 6.1.1]{FOT}). By  \cite[Theorem 6.1.1]{FOT},
we have the Feynman-Kac representation formula
\begin{equation}
\label{ab-fk_1}
E^{D,\mu}_x f(X_t)=E_xe^{-A^\mu_t}f(X_t)\mathbf{1}_{\{t<\tau_D\}},\quad x\in D,\, t\ge 0.
\end{equation}

\section{Integral supersolutions of the obstacle problem}

Let $h:D\rightarrow \mathbb R\cup\{\infty\}$ be a  measurable
strictly positive  function  and $f:\mathbb R\rightarrow \mathbb
R$ be a continuous function.
Consider the following obstacle problem:
\begin{equation}
\label{eq3.1} \left\{
\begin{array}{l}\max\big\{-(\Delta^{\alpha/2})_{|D} u-f(u),u- h\big\}=0,\medskip\\
u>0\quad \mbox{on}\quad D.
\end{array}
\right.
\end{equation}
In this section, we give a definition of an {\em integral solution}  to (\ref{eq3.1}).
We also recall widely used in the literature definition of {\em weak solutions} via variational inequalities.
We show that both definitions are equivalent under regularity assumptions on $u$. 
However, we prefer to work on  integral solutions.
The advantage of this notion is that its formulation  does not require  any regularity from  solutions (merely integrability),
and it allows us introduce the notion of super and subsolutions. Moreover, by using integral solutions, we may apply
probabilistic potential theory, which is very useful in many reasonings. 
We then show  that minimum
of two {\em integral supersolutions} of (\ref{eq3.1}) is again an integral  supersolution.
This  property will be one of the crucial ingredients in the proof of
uniqueness  of (\ref{eq1.1}).

\begin{definition}
\label{df3.1} We say that a quasi-continuous function $u\in
L^1(D;m)$  is an {\em  integral solution} to (\ref{eq3.1}) if there exists a
nonnegative $\nu\in \mathbb M_{0}(D)$ (we call it the reaction
measure for $u$) such that
\begin{enumerate}
\item[\rm{(a)}] $0<u\le h$ $m$-a.e. and  $R^D|f(u)|<\infty$ q.e.,
\item[\rm{(b)}] For q.e. $x\in D$,
\[
u(x)=R^Df(u)(x)-R^D\nu(x),
\]
\item[\rm{(c)}] For any quasi-continuous function $\eta$ on $D$ 
such that $u\le \eta\le h$ $m$-a.e.,
\[
\int_D(\eta-u)\,d\nu= 0.
\]
\end{enumerate}
\end{definition}

\begin{remark}
\label{rm3.1} (i) Take any  quasi-continuous function $\hat h$ such that $u\le \hat h\le
h$ $m$-a.e. Observe that  $u$ is an  integral solution to (\ref{eq3.1}) if and only
if $u$ is an  integral solution to (\ref{eq3.1}) with $h$ replaced by $\hat
h$.
\smallskip\\
(ii) If $h\in D(\EE_D)$ is quasi-continuous, then $u$ is an integral solution to
(\ref{eq3.1}) if and only if  (a), (b) are satisfied  and
$\int_D(h-u)\,d\nu=0$. \smallskip\\
(iii) Recall that by \cite[Lemma 2.1.4]{FOT}, if $u\le \eta,\, m$-a.e. and $u,\eta$ are quasi-continuous, then $u\le \eta$  q.e. \smallskip\\
(iv) Note that by \cite[Theorem 4.6.1,Theorem A.2.7]{FOT}, functions $R^Df^+(u)$, $R^Df^-(u)$, $R^D\nu$ are quasi-continuous as all the mentioned functions are excessive and finite q.e. So, if (b) holds $m$-a.e., then by (iii),
(b) holds q.e. Clearly, $R^Df^+(u), R^Df^-(u)$ does not depend on the version of $u$. Therefore, if we assume in the Definition \ref{eq3.1} that $u$ is merely a measurable function such that (a), (c) hold, and (b) holds $m$-a.e., then there exists a version $\tilde u$ of $u$ such that $\tilde u$ is quasi-continuous and (b) holds q.e. with $u$ replaced by $\tilde u$: this version is given by $\tilde u(x):= R^Df(u)(x)-R^D\nu(x)$ if  $R^D|f|(u)(x)+R^D\nu(x)<\infty$ and zero otherwise. Clearly, $\tilde u= R^Df(\tilde u)-R^D\nu$ q.e.

\end{remark}

Recall that $(\cdot,\cdot)$ stands for the usual inner product in $L^2(D;m)$.

\begin{definition}
\label{df3.2} We say that a quasi-continuous function $u\in D(\EE_D)$  is a {\em   weak solution} to (\ref{eq3.1}) 
if 
\begin{enumerate}
\item[\rm{(a)}] $0<u\le h$ $m$-a.e. and $f(u)\in L^2(D;m)$,
\item[\rm{(b)}] For every $\eta\in D(\EE_D)$ 
such that $\eta\le h$ $m$-a.e.,
\[
\EE_D(u,\eta-u)\ge (f(u),\eta-u).
\]
\end{enumerate}
\end{definition}

\begin{proposition}
\label{prop4.4ext1}
Let $u\in D(\EE_D)$ and $f(u)\in L^2(D;m)$. Then   $u$ is an integral solution to {\rm{(\ref{eq3.1})}} iff $u$ is a weak  solution to {\rm{(\ref{eq3.1})}}. 
\end{proposition}
\begin{proof}
Let $u$ be an integral solution to (\ref{eq3.1}). Since $f(u)\in L^2(D;m)$, $R^Df(u)\in D((\Delta^{\alpha/2})_{|D})\subset D(\EE_D)$.
Thus, $\nu\in S_0(D)$. Therefore, by (\ref{eq.smfe1}) and (c) of Definition \ref{df3.1}, for every $\eta\in D(\EE_D)$ such that $\eta\le h\, m$-a.e.,
\[
\EE_D(u,\eta-u)=\EE_D(R^Df(u)-R^D\nu,\eta-u)=(f(u),\eta-u)-\int_D(\tilde\eta-\tilde u)\,d\nu\ge (f(u),\eta-u).
\] 
Now, let  $u$ be a weak  solution to (\ref{eq3.1}). By \cite[Theorem 5.2, Chapter 3]{Lions} $u_n\rightarrow u$ weakly in $\EE_D$, where $u_n\in D(\EE_D)$
is a unique  solution to the following variational equality
\[
\EE_D(u_n,\eta)=(f(u),\eta)-(n(u_n-h)^+,\eta),\quad \eta\in D(\EE_D).
\]
Clearly, $u_n=R^Df(u)-R^D\nu_n$, $m$-a.e., where $\nu_n=n(u_n-h)^+\cdot m$. By \cite[Theorem 3.8]{K:SM2}, $u_n\searrow w$, where $w$ is an integral solution to 
\begin{equation}
\label{eq3.1ext} 
\max\big\{-(\Delta^{\alpha/2})_{|D} w-f(u),w- h\big\}=0.
\end{equation}
Since $u=w,\, m$-a.e., we get the desired result (cf. Remark \ref{rm3.1}(iv)).
\end{proof}


\begin{definition}
\label{def3.2}
We say that a quasi-continuous function $u\in L^1(D;m)$ is an {\em integral  supersolution} (resp. {\em integral subsolution}) to (\ref{eq3.1}) if there exists a nonnegative measure $\nu\in \mathbb M_{0}(D)$ and a nonnegative  (resp. nonpositive) measure $\mu\in\mathbb M_{0}(D)$  such that conditions
(a) and (c) of Definition \ref{df3.1} are satisfied, and  moreover, for q.e. $x\in \BR^d$,
\[
u(x)=R^Df(u)(x)+R^D\mu(x)-R^D\nu(x).
\]
\end{definition}

\begin{proposition}
\label{prop3.1}
If $u$ is an integral supersolution and an integral subsolution to {\rm{(\ref{eq3.1})}}, then $u$ is an integral solution to \rm{(\ref{eq3.1})}.
\end{proposition}
\begin{proof}
By assumptions and Definition \ref{def3.2}, there exist nonnegative measures $\nu_1,\nu_2,\mu_1,\mu_2\in\mathbb M_0(D)$    such that
\[
u(x)=R^Df(u)(x)+R^D\mu_1(x)-R^D\nu_1(x),
\]
\[
u(x)=R^Df(u)(x)-R^D\mu_2(x)-R^D\nu_2(x),
\]
and $\int_D(\eta-u)\,d\nu_1=\int_D(\eta-u)\,d\nu_2=0$ for every quasi-continuous $\eta$
on $D$ such that $u\le \eta\le h$ $m$-a.e. Thus,
\[
R^D\nu_1(x)-R^D\nu_2(x)=R^D\mu_1(x)+R^D\mu_2(x)
\]
for q.e. $x\in D$. From this, we conclude that
$\nu_1-\nu_2=\mu_1+\mu_2$.
Therefore,  there exist nonnegative  $\alpha,\beta,\gamma\in\BB(D)$ such that $\alpha+\beta=1, \gamma\le 1$ and
\[
\nu_2=\gamma\cdot\nu_1,\quad \mu_1=\alpha(1-\gamma)\cdot\nu_1,\quad \mu_2=\beta(1-\gamma)\cdot\nu_1.
\]
Consequently,  for q.e $x\in D$,
\[
u(x)=R^Df(u)(x)-R^D((\beta+\alpha\gamma)\cdot\nu_1)(x).
\]
Since $\int_D(\eta-u)(\beta+\alpha\gamma)\,d\nu_1=0$ for every quasi-continuous $\eta$
on $D$ such that $u\le \eta\le h$ $m$-a.e., we see that $u$ is an integral solution to (\ref{eq3.1}).
\end{proof}

\begin{proposition}
\label{prop3.2}
Let $u_1, u_2$ be  integral supersolutions to {\rm{(\ref{eq3.1})}}. Then $u_1\wedge u_2$ is a an integral supersolution to \rm{(\ref{eq3.1})}.
\end{proposition}
\begin{proof}
By  the definition of a an integral supersolution to (\ref{eq3.1}) and Lemma \ref{lm2.4.1}, for q.e. $x\in D$, we have
\[
u_i(X_t)=\int_t^{\tau_D}\,dA^{\mu_i}_r+\int_t^{\tau_D}f(u(X_r))\,dr
-\int_t^{\tau_D}\,dA^{\nu_i}_r-\int_t^{\tau_D}\,dM^i_r,\quad t\in [0,\tau_D],\quad P_x\mbox{-a.s.,}
\]
$i=1,2$, for some  nonnegative $\nu_1,\nu_2,\mu_1,\mu_2\in\mathbb M_0(D)$ and
some uniformly integrable MAFs $M^1,M^2$ of $\mathbb X^D$. By the Tanaka-Meyer formula (see, e.g., \cite[IV.Theorem 70]{Protter}) applied 
to $u_1(X)-(u_1(X)-u_2(X))^+=u_1(X)\wedge u_2(X)$,
there exists an increasing c\`adl\`ag process $C$ with $C_0=0$ such that for q.e. $x\in D$,
\begin{align}
\label{eq4.ext.56}
\nonumber(u_1\wedge u_2)(X_t)&=(u_1\wedge u_2)(x)-C_t\\
&\nonumber\,\,\,-\int_0^t \mathbf{1}_{\{u_1(X_r)>u_2(X_r)\}}f(u_2(X_r))\,dr
-\int_0^t \mathbf{1}_{\{u_1(X_r)>u_2(X_r)\}}\,(dA^{\mu_2}_r-dA^{\nu_2}_r)\\
&\nonumber\,\,\, -\int_0^t \mathbf{1}_{\{u_2(X_r)\ge u_1(X_r)\}}f(u_1(X_r))\,dr-\int_0^t \mathbf{1}_{\{u_2(X_r)\ge u_2(X_r)\}}\,(dA^{\mu_1}_r-dA^{\nu_1}_r)\\
&\nonumber\,\,\,+\int_0^t \mathbf{1}_{\{u_1(X_{r-})>u_2(X_{r-})\}}\,dM^2_r+\int_0^t \mathbf{1}_{\{u_2(X_{r-})\ge u_2(X_{r-})\}}\,dM^1_r\\
&\nonumber=(u_1\wedge u_2)(x)-C_t-\int_0^tf((u_1\wedge u_2)(X_r))\,dr\\
&\nonumber\,\,\,-\int_0^t \mathbf{1}_{\{u_1(X_r)>u_2(X_r)\}}\,dA^{\mu_2}_r+\int_0^t \mathbf{1}_{\{u_2(X_r)\ge u_2(X_r)\}}\,dA^{\mu_1}_r\\
&\nonumber\,\,\,+\int_0^t \mathbf{1}_{\{u_1(X_r)>u_2(X_r)\}}\,dA^{\nu_2}_r+\int_0^t \mathbf{1}_{\{u_2(X_r)\ge u_2(X_r)\}}\,dA^{\nu_1}_r\\
&\,\,\,+\int_0^t \mathbf{1}_{\{u_1(X_{r-)}>u_2(X_{r-})\}}\,dM^2_r+\int_0^t \mathbf{1}_{\{u_2(X_{r-})\ge u_2(X_{r-})\}}\,dM^1_r,
\end{align}
$t\in [0,\tau_D],\,P_x\mbox{-a.s.}$ From the above formula, we get, in particular, that $C$ is a positive  AF of $\mathbb X^D$. 
By \cite[Theorem A.3.16]{FOT} there exists a positive AF  $C^p$ which is the dual predictable projection of $C$ under measure $P^D_x$ for q.e. $x\in D$. 
Since $\mathbb X^D$ is a Hunt process it   is, by the very definition,  quasi-left continuous, so it has only {\em totally inaccessible jumps}. Therefore, since $u_1\wedge u_2$ is quasi-continuous, process $(u_1\wedge u_2)(X)$ has only totally inaccessible jumps
under measure $P^D_x$ for q.e. $x\in D$ (see \cite[Theorem 4.2.2]{FOT}).  Moreover, by \cite[Proposition 2, Proposition 4]{CW} every local $(\FF_t)_{t\ge 0}$-martingale has only totally inaccessible   jumps. By the definition of dual predictable projection, $C^p-C$ is an $(\FF_t)_{t\ge 0}$-martingale. Therefore, by (\ref{eq4.ext.56}), $C^{p}$ has only totally inaccessible  jumps under measure $P^D_x$ for q.e. $x\in D$. However, $C^{p}$ is predictable. Consequently,   $C^{p}$ is continuous. So, $C^p$ is a PCAF of $\mathbb X^D$. Hence
$C^p=A^\beta$ for some  nonnegative $\beta\in \mathbb M_{0}(D)$ (cf. Section \ref{sec2.4}). Define  $\mu= \mathbf{1}_{u_1>u_2}\cdot\mu_2+\mathbf{1}_{u_2\ge u_1}\cdot\mu_1+\beta$ and $\nu= \mathbf{1}_{u_1>u_2}\cdot\nu_2+\mathbf{1}_{u_2\ge u_1}\cdot\nu_1$. By Lemma \ref{lm2.4.1}, for q.e $x\in D$,
\begin{equation}
\label{eq4.ext13}
(u_1\wedge u_2)(x)=R^Df(u_1\wedge u_2)(x)+R^D\mu(x)-R^D\nu(x).
\end{equation}
Let  $\eta$ be a  quasi-continuous function such that $u_1\wedge u_2\le \eta\le h$ $m$-a.e.
Observe that
\begin{align*}
\int_D(\eta-u_1\wedge u_2)\,d\nu&=\int_{u_1>u_2}(\eta-u_2)\,d\nu_2+\int_{u_1\le u_2}(\eta-u_1)\,d\nu_1
\\&\le \int_{u_1>u_2}(\eta\vee u_2-u_2)\,d\nu_2+\int_{u_1\le u_2}(\eta\vee u_1-u_1)\,d\nu_1
\\&\le \int_D(\eta\vee u_2-u_2)\,d\nu_2+\int_D(\eta\vee u_1-u_1)\,d\nu_1.
\end{align*}
Clearly, $u_2\le \eta\vee u_2\le h$ $m$-a.e., and $u_1\le \eta\vee u_1\le h$ $m$-a.e. So, by condition (c) of Definition \ref{df3.1}
applied to $(u_1,\nu_1)$ and $(u_2,\nu_2)$, we get 
\[
\int_D(\eta\vee u_2-u_2)\,d\nu_2+\int_D(\eta\vee u_1-u_1)\,d\nu_1=0.
\]
Thus, $\int_D(\eta-u_1\wedge u_2)\,d\nu=0$. This combined with (\ref{eq4.ext13}) implies that $u_1\wedge u_2$ is an integral supersolution to (\ref{eq3.1}).
\end{proof}

\begin{proposition}
\label{prop3.3}
Assume  that $f$ is  nondecreasing. Let $\underline u$ (resp. $\overline u$) be  a bounded integral subsolution (resp.  supersolution) to {\rm{(\ref{eq3.1})}} and $\underline u\le \overline u$ q.e. Then there exists an integral solution $u$ to {\rm{(\ref{eq3.1})}}
such that $\underline u\le u\le \overline u $ q.e.
\end{proposition}
\begin{proof}
Let $u_0=\underline u$. We first show  that  for each $n\ge1$ there exists  an integral solution $u_n$ to the problem
\begin{equation}
\label{eq3.2} \left\{
\begin{array}{l}\max\big\{-(\Delta^{\alpha/2})_{|D} v-f(u_{n-1}),v- h\big\}=0,\medskip\\
v>0\quad \mbox{on}\quad D,
\end{array}
\right.
\end{equation}
and 
\[
u_n\le u_{n+1},\mbox{q.e.},\qquad \nu_n\le \nu_{n+1}\qquad \underline u\le u_n\le\overline u,\,\mbox{q.e.}\quad n\ge 1,
\]
where $\nu_n$ is the reaction measure for $u_n$.
Indeed, the existence of $u_1$ follows from \cite[Theorem 3.8]{K:SM2}. By \cite[Proposition 3.12]{K:SM2}, $\underline u\le u_1\le \overline u$ q.e. In particular,
$R^D|f(u_1)|<\infty$ q.e. Hence, by \cite[Theorem 3.8]{K:SM2} again, there exists an integral solution $u_2$ to (\ref{eq3.2}), and  by \cite[Proposition 3.12]{K:SM2} again, $\underline u\le u_2\le  \overline u$ q.e. and  $u_1\le u_2$ q.e. Continuing in this fashion, we get the existence  of $\{u_n\}$ having the desired properties.
Moreover, by \cite[Proposition 3.12]{K:SM2}, $\nu_n\le \nu_{n+1},\, n\ge 1$. Let $u=\sup_{n\ge 1} u_n$, and
\[
\nu(B):=\lim_n\nu_n(B),\quad B\in\BB(D).
\]
Observe that 
\[
R^D\nu_n\le R^Df(u_{n-1})\le R^D|f(\overline u)|+R^D|f(\underline u)|,\quad \mbox{q.e.}
\]
Therefore, by \cite[Lemma 5.4]{KR:JFA},
\[
\|\nu_n\|_{TV}\le \|f(\underline u)\|_{L^1}+\|f(\overline u)\|_{L^1},\quad n\ge 1.
\]
The right-hand side is finite since $\underline u, \overline u$ are bounded and $D$ is bounded.
By the Vitali-Hahn-Saks theorem $\nu$ is a bounded Borel measure and 
\begin{equation}
\label{eq3.ext11}
\int_Df\,d\nu_n\rightarrow \int_Df\,d\nu,\quad f\in L^1(D;\nu).
\end{equation}
By (\ref{eq2.3.1}), $R^D\nu<\infty$ q.e., and $R^D|f(\overline u)|+R^D|f(\underline u)|<\infty$ q.e.
So, by (\ref{eq3.ext11}) and the Lebesgue dominated convergence theorem,
\[
R^D\nu_n\rightarrow R^D\nu,\quad R^Df(u_{n-1})\rightarrow R^Df(u),\quad \mbox{q.e.}
\]
Thus,
\[
u=R^Df(u)-R^D\nu,\quad\mbox{q.e.}
\]
By Remark \ref{rm3.1}(iv), $u$ is quasi-continuous. Clearly, $u\ge h,\, m$-a.e. Let $\eta$ be an arbitrary quasi-continuous function such that $u\le \eta \le h$ $m$-a.e. Then, by the minimality condition (c)
of Definition \ref{df3.1} applied to $(u_n,\nu_n)$, and by  (\ref{eq3.ext11}),
\[
0=\int_D(\eta-u_n)\,d\nu_n\ge \int_D(\eta-u)\,d\nu_n\rightarrow \int_D(\eta-u)\,d\nu\ge 0.
\]
Thus, $u$ is an integral solution to (\ref{eq3.1}).
\end{proof}

\section{Uniqueness result}

As in Sections 2 and 3, we assume that $D$ is a bounded Lipschitz domain in $\BR^d$ ($d\ge 2$) and $D_0\subset D$ is a bounded Dirichlet regular domain. We assume that $a>\lambda_1^D$, where $\lambda_1^D$ is the first eigenvalue for the operator $-(\Delta^{\alpha/2})_{|D}$. By $\varphi_1^D$
we denote the ground state for $-(\Delta^{\alpha/2})_{|D}$, i.e. a unique strictly positive  function $\varphi_1^D\in D(\EE_D)$ such that $\|\varphi_1^D\|_{L^2}=1$ and
\[
\EE_D(\varphi_1^D,\eta)=\lambda_1^D(\varphi_1^D,\eta),\quad \eta\in D(\EE_D).
\]
It is well known that $\lambda_1^D>0$, and by the regularity of $D$, $\varphi_1^D\in C_0(D)$.

To  prove a uniqueness result for (\ref{eq1.1}), we shall need some regularity results for integral solutions to (\ref{eq1.1}), and
the  result  which compare  the  Green function $G_D$ for  $(\Delta^{\alpha/2})_{|D}$ with the Green function $G^\nu_D$
for $(\Delta^{\alpha/2})_{|D}-\nu$.

It is well known (see e.g. \cite{Grzywny}) that the semigroup $(P^D_t)_{t>0}$ is intrinsically ultracontractive, which implies that for every $t>0$
there exist  constants $c_1(t),\, c_2(t)>0$ such that
\begin{equation}
\label{eq4.1}
c_1(t)\varphi_1^D(x)\varphi_1^D(y)\le p_D(t,x,y)\le c_2(t)\varphi_1^D(x)\varphi_1^D(y),\quad x,y\in D.
\end{equation}

\begin{proposition}
\label{prop4.2ms}
Let $u$ be an integral solution to {\rm{(\ref{eq1.1})}} and $\nu$ be the reaction measure for $u$.
Then $u\cdot\nu=\nu$.
\end{proposition}
\begin{proof}
Set $w=R^Du,\, w_n=R^D(u\wedge n)$, and
\begin{equation}
\label{eq4.1.0}
h(x)=1+a\rho(x),\quad \rho(x)=E_x\int_0^{\tau_{D_0}}u(X_r)\,dr,\quad x\in D.
\end{equation}
Clearly, $w_n\in D(\EE_D)$ since $u\wedge n\in L^2(D;m)$.
By Dynkin's formula (see \cite[(4.4.2)]{FOT}),
\begin{equation}
\label{eq4.1.0.ext}
\rho(x)+E_x [w(X_{\tau_{D_0}})]=R^Du(x),\quad x\in D.
\end{equation}
By \cite[Lemma 4.3.1]{FOT}, $E_\cdot [w_n(X_{\tau_{D_0}})]$ is an excessive function for $n\ge 1$.
So, by \cite[Proposition II.2.2]{BG}, $E_\cdot [w(X_{\tau_{D_0}})]$ is an excessive function.  By the definition of an integral solution, $u\in L^1(D;m)$.
Therefore, by (\ref{eq2.3.1}), $R^Du$ is finite q.e. So, by (\ref{eq4.1.0.ext}), $E_\cdot [w(X_{\tau_{D_0}})]$ is finite q.e.
Consequently, $\rho$ as a difference of excessive functions
finite q.e. is quasi-continuous (see comments in Remark \ref{rm3.1}(iv)). Thus,  $h$ is quasi-continuous. 
By regularity of $D_0$, $\rho(x)=0,\, x\in D\setminus D_0$. Moreover, 
\begin{equation}
\label{eq4.1.0.ext1}
\rho(x)=\int_D G_{D_0}(x,y)u(y)\,m(dy)>0,\quad x\in D_0.
\end{equation}
By Dynkin's formula again
\[
u(x)=E_xu(X_{\tau_{D_0}})+aE_x\int_0^{\tau_{D_0}} u(X_r)\,dr-E_x\int_{0}^{\tau_{D_0}}\,dA^\nu_r\le 1+a\rho(x)=h(x),\quad\mbox{q.e.}
\]
From this and the fact that $\rho=0$ on $D\setminus D_0$, we get that the $n$-th power of $h$ satisfies
\[
u\le h^n\le\mathbb I_{D\setminus \overline D_0},\quad n\ge 1,\quad\mbox{q.e.}
\]
Hence, by the definition of an integral solution to the obstacle problem,
\begin{equation}
\label{eq4.1.0.0}
\int_D(h^n-u)\,d\nu=0,\quad n\ge 1.
\end{equation}
By (\ref{eq4.1.0.ext1}),  for every $x\in D_0$, $h^n(x)\nearrow \infty$ as $n\rightarrow\infty$. It follows from (\ref{eq4.1.0.0}) that  supp$[\nu]\subset D\setminus D_0$.  By this and (\ref{eq4.1.0.0}) again, for every nonnegative $\eta\in C_c(D)$ we have
\[
0=\int_D\eta(h-u)\,d\nu=\int_{D\setminus D_0}\eta(h-u)\,d\nu=\int_{D\setminus D_0}\eta(1-u)\,d\nu=\int_{D}\eta(1-u)\,d\nu,
\]
which implies the desired  result.
\end{proof}

\begin{proposition}
\label{prop4.1}
If  $u$ is an integral solution to {\rm{(\ref{eq1.1})}}, then
\begin{enumerate}
\item[\rm{(i)}] $u(x)\le c\|u\|_{L_1(D;m)}\varphi_1^D(x)$ for q.e. $x\in D$.
\item[\rm{(ii)}] $u\in C_0(D)$.
\item[\rm{(iii)}] The reaction measure  $\nu$ for $u$ is bounded and $\|\nu\|_{TV}\le a\|u\|_{L^1(D;m)}$.
\end{enumerate}
\end{proposition}
\begin{proof}
By the definition of an integral solution to (\ref{eq1.1}) and Proposition \ref{prop4.2ms},
\[
u(x)=aE_x\int_0^{\tau_D}u(X_r)\,dr-E_x\int_0^{\tau_D}u(X_r)\,dA^\nu_r
\]
for q.e $x\in D$. Using  Lemma \ref{lm2.4.1} and the integration by part formula applied to the product
$e^{at-A^\nu_t}u(X_t)$ yields
\[
u(x)=e^{at}E_xe^{-A^\nu_t}u(X_t),\quad t\ge 0,
\]
for q.e. $x\in D$. Therefore, by the ultracontractivity of $(P^D_t)_{t\ge 0}$, for every $t>0$,
\[
u(x)\le e^{at}\int_Dp_D(t,x,y)u(y)\,dy\le c_te^{at}\varphi_1^D(x)\|u\|_{L^1(D;m)}
\]
for q.e. $x\in D$, which proves  (i). To prove  (ii), consider  the function $h$ defined by (\ref{eq4.1.0}).
By regularity of the set $D_0$ and  \cite{Chung}, $(P^{D_0}_t)_{t\ge 0}$ is doubly Feller (cf. Section \ref{sec2.4}). Therefore, by (i) $\rho\in C(D_0)$. By regularity of $D_0$, in fact $\rho\in C_0(D_0)$. Thus,  $h\in C(\overline D)$. By the proof of Proposition \ref{prop4.2ms},
\[
u\le h\le \mathbb I_{D\setminus \overline D_0},\quad \mbox{q.e.}
\]
So, by Remark \ref{rm3.1}(i), $u$ is an integral solution
to (\ref{eq1.1}) with $ \mathbb I_{D\setminus \overline D_0}$ replaced by $h$. Therefore, $u\in C_0(D)$
by \cite[Theorem 1]{Stettner}, which proves (ii).
By the definition of an integral solution to (\ref{eq1.1}),
\begin{equation}
\label{eq4.1.01}
E_x\int_0^{\tau_D}\,dA^\nu_r\le aE_x\int_0^{\tau_D} u(X_r)\,dr
\end{equation}
for q.e. $x\in D$.
From (\ref{eq4.1.01})  and \cite[Lemma 5.4]{KR:JFA} we get (iii).
\end{proof}

\begin{proposition}
\label{prop4.2}
Let $u$ be an integral solution to {\rm{(\ref{eq1.1})}} and $\nu$ be its reaction measure.
Then
\begin{enumerate}
\item[\rm{(i)}] supp$[\nu]$ is a compact subset of $D$.
\item[\rm{(ii)}] $\sup_{x\in D}R^D\nu(x)<\infty$.
\end{enumerate}

\end{proposition}
\begin{proof}
By Proposition \ref{prop4.2ms},
supp$[\nu]\subset \{u=1\}$, which when combined with Proposition \ref{prop4.1}(ii) implies (i).
Assertion  (ii) follows easily from (\ref{eq4.1.01}), Proposition \ref{prop4.1} and (\ref{eq2.4.4}).
\end{proof}

\begin{proposition}
\label{prop4.4}
Let $u$ be an integral solution to {\rm{(\ref{eq1.1})}} and let $\nu$ be its reaction measure. Then $u\in D(\EE_D)$, $\nu\in S_0(D)$ and for every $\eta\in D(\EE_D)$,
\begin{equation}
\label{eq4.4}
\EE_D(u,\eta)=a(u,\eta)-\int_D\tilde\eta\,d\nu.
\end{equation}
\end{proposition}
\begin{proof}
By Proposition \ref{prop4.1}, $u\in L^1(D;m)$ and $\nu\in \MM_{0,b}(D)$. Therefore, by \cite[Theorem 3.5]{KR:NoDEA},  $u$ is a renormalized  solution to (\ref{eq1.1}) in the sense defined in  \cite{KR:NoDEA}. By the definition of a renormalized  solution,  $u\wedge k\in D(\EE_D)$, $k\ge 1$, which when combined with  Proposition \ref{prop4.1}(i) implies that $u\in D(\EE_D)$. 
Moreover, by the definition of a renormalized  solution, there exists a sequence $\nu_k\subset \MM_{0,b}(D)$ such that $\|\nu_k\|_{TV}\rightarrow 0$ and for every bounded $\eta\in D(\EE_D)$,
\[
\EE_D(u\wedge k,\eta)+\int_D\tilde\eta\,d\nu=a(u,\eta)+\int_D\tilde\eta\,d\nu_k.
\]
Letting $k\rightarrow \infty$ yields (\ref{eq4.4}) for every bounded $\eta\in D(\EE_D)$. Applying now a simple approximation argument, we get that $\nu\in S_0(D)$, and (\ref{eq4.4}) holds for any $\eta\in D(\EE_D)$.
\end{proof}

Let $w$ be a strictly positive excessive function with respect to $(P^D_t)_{t\ge 0}$.
We say that $G_D$ has  $w$-triangle property iff there exists $C>0$ such that
\begin{equation}
\label{eq4.wtp1}
G_D(x,z)G_D(z,y)\le CG_D(x,y)\max\Big(\frac{w(z)}{w(x)}G_D(x,z),\frac{w(z)}{w(y)}G_D(z,y)\Big),\quad x,y,z\in D.
\end{equation}
The above notion was introduced in \cite{Hansen}.  Observe that if we set 
\[
\rho(x,y):=\frac{w(x)w(y)}{G_D(x,y)},\quad x,y\in D,
\]
then \eqref{eq4.wtp1} is equivalent to
\begin{equation}
\label{eq4.wtp1abc}
\rho(x,y)\le C\max\big(\rho(x,z),\rho(z,y)\big),\quad x,y,z\in D.
\end{equation}
Clearly, $\rho(x,y)=0$ if and only if $x=y$, and $\rho(x,y)=\rho(y,x)$. Therefore, $w$-triangle property
is equivalent to the statement that $\rho$ is a quasi-metric on $D$.

Recall here that at the beginning of Section \ref{sec2}, we introduced
$\lambda>0$ and $r_0>0$. We shall show that $G_D$ has $w$-triangle property for
Lipschitz domains, where $w=\phi$ and 
\[
\phi(x)=\min(G_D(x,x_0),c_{d,\alpha}(r_0/4)^{\alpha-d}),\quad x\in D
\]
for a fixed $x_0\in D$. Here $c_{d,\alpha}=\frac{\Gamma((d-\alpha)/2)}{2^\alpha\pi^{d/2}|\Gamma(\alpha/2)|}$.
Set $\kappa= 1/2\sqrt{1+\lambda^2}$, and fix $x_1\in D$ such that $|x_0-x_1|=r_0/4$. We let $\delta_D(x):= dist(x,\partial D),\, x\in D$.

\begin{proposition}
\label{prop4.wtp1}
Green function $G_D$ has  $\phi$-triangle property.
\end{proposition}
\begin{proof}
Let $x,y,z,z_1,z_2\in D$.
By \cite[Theorem 1]{Jakubowski}, 
\begin{equation}
\label{eq4.wtp2}
c_1\frac{\phi(x)\phi(z_i)}{\phi^2(A_i)}\le G_D(x,z_i)\le c_2\frac{\phi(x)\phi(z_i)}{\phi^2(A_i)},
\end{equation}
for $A_i\in\mathfrak B(x,z_i), i=1,2$, where $c_1,c_2$ depend only on $D,d,\lambda,\alpha$. Here, for $r_i:= \delta_D(x)\vee \delta_D(z_i)\vee|x-z_i|\le r_0/32$,
\[
\mathfrak B(x,z_i)=\{A\in D: B(A,\kappa r_i)\subset D\cap B(x,3r_i)\cap B(z_i,3r_i) \},
\]
and $\mathfrak B(x,z_i)=\{x_1\}$ for $r_i>r_0/32$. Suppose that $|x-z|\ge |y-z|$. Set
$z_1:=y, z_2:= z$. By (\ref{eq4.wtp2}),
\[
\frac{G_D(x,z)}{G_D(x,y)}=\frac{G_D(x,z_2)}{G_D(x,z_1)}\le \frac{c_2}{c_1}\frac{\phi(z_2)}{\phi(z_1)}\Big(\frac{\phi(A_1)}{\phi(A_2)}\Big)^2
=\frac{c_2}{c_1}\frac{\phi(z)}{\phi(y)}\Big(\frac{\phi(A_1)}{\phi(A_2)}\Big)^2.
\] 
Observe that $|x-z_1|\le 2|x-z_2|$. Therefore, by \cite[Lemma 13]{Jakubowski}, there exists $C_1>0$ such that $\frac{\phi(A_1)}{\phi(A_2)}\le C_1$,
where $C_1$ depends only on $\alpha$, $d$ and  $D$. Consequently, there exists $C>0$ such that
\[
\frac{G_D(x,z)}{G_D(x,y)}\le C\frac{\phi(z)}{\phi(y)},\quad x,y,z\in D;\,  |x-z|\ge |y-z|.
\]
Equivalently,
\[
\rho(x,y)\le C\rho(x,z),\quad x,y,z\in D;\,  |x-z|\ge |y-z|.
\]
Hence we get easily (\ref{eq4.wtp1abc}).
\end{proof}

Let $G$ denote the Green function for $\BR^d$ (and the operator $\Delta^{\alpha/2}$). It is well known that there is $c>0$ such that
\[
G(x,y)=\frac{c}{|x-y|^{d-\alpha}}\,,\quad x,y\in\BR^d.
\]

\begin{lemma}
\label{lm4.1}
Assume that $\mu\in\MM_{0,b}(D)$ is nonnegative, $K:=\mbox{supp}[\mu]$ is a compact subset of $D$
and $\sup_{x\in D}R^D\mu(x)<\infty$. Then $G_D^\mu\sim G_D$, i.e. there exist $c_1, c_2>0$ such that $c_1G_D^\mu\le  G_D\le c_2 G_D^\mu$ on $D\times D$.
\end{lemma}
\begin{proof}
Let $\bar \mu $ be an extension of $\mu$ to $\BR^d$ defined as $\bar\mu (B)= \mu(D\cap B)$ for any Borel set $B\subset\BR^d$.
Let $V$ be an open  set such that $K\subset V\subset\overline V\subset D$, and let $r=\mbox{dist}(K,\partial V)$.
By \cite[Theorem 1.2]{ChenSong}, there is $c>0$ such that
\[
cG(x,y)\le G_D(x,y),\quad x,y\in V.
\]
From this and the assumptions of the lemma it follows that
\begin{align}
\label{eq4.2}
\nonumber\sup_{x\in D}R\bar\mu(x)&\le\max\{\sup_{x\in V}R\bar\mu(x),\sup_{x\in D\setminus V}R\bar\mu(x)\}\\&
\nonumber \le
\max\Big\{c\sup_{x\in V}R^D\mu(x),\sup_{x\in D\setminus V}\int_K G(x,y)\,\mu(dy)\Big\}\\
&\le c\max\{\sup_{x\in D}R^D\mu(x),r^{\alpha-d}\|\mu\|_{TV}\}.
\end{align}
Next, for all $x,y\in D$,
\begin{align}
\label{eq4.3}
\nonumber \int_D\frac{G(x,z)G(z,y)}{G(x,y)}\,\mu(dz)
&=\int_D\frac{|x-y|^{d-\alpha}}{|x-z|^{d-\alpha}|z-y|^{d-\alpha}}\,\mu(dz)\\
&\nonumber
\le 2^{d-\alpha}\int_D\frac{\max\{|x-z|^{d-\alpha},|z-y|^{d-\alpha}\}}
{|x-z|^{d-\alpha}|z-y|^{d-\alpha}}\,\mu(dz)\\&
\le 2^{d-\alpha}R\bar\mu(x)+2^{d-\alpha}R\bar\mu(y).
\end{align}
By the 3G Theorem (see \cite{KimLee}),
\[
\int_D\frac{G_D(x,z)G_D(z,y)}{G_D(x,y)}\,\mu(dz)\le c \int_D\frac{G(x,z)G(z,y)}{G(x,y)}\,\mu(dz),\quad x,y\in D.
\]
This when combined with (\ref{eq4.2}) and (\ref{eq4.3}) shows  that there exists $C>0$ such that
\begin{equation}
\label{eq.qe.ge}
 \int_DG_D(x,z)G_D(z,y)\,\mu(dz)\le C G_D(x,y),\quad x,y\in D.
\end{equation}
From this we conclude that for every nonnegative Borel measure $\nu$ on $D$,
\[
 \int_DG_D(x,z)\Big(\int_DG_D(z,y)\,\nu(dy)\Big)\,\mu(dz)\le C \int_DG_D(x,y)\,\nu(dy),\quad x,y\in D.
\]
Equivalently,
\begin{equation}
\label{eq4.4abcd}
 \int_DG_D(x,z)R^D\nu(z)\,\mu(dz)\le C R^D\nu(x),\quad x,y\in D.
\end{equation}
It is well known (see \cite[Theorem 17, page 230]{DellacherieMeyer}) that each $(P^D_t)_{t\ge 0}$-excessive function is an increasing limit of functions
of the form $R^D\nu$ for some nonnegative Borel measure $\nu$. Therefore from (\ref{eq4.4abcd}) it follows that  for every excessive function $e$,
\[
\int_DG_D(x,z)e(z)\,\mu(dz)\le C e(x),\quad x,y\in D.
\]
Taking $e=\phi$ (it is excessive as a minimum of excessive functions), we get
\[
\sup_{x\in D}\int_D\frac{\phi(z)}{\phi(x)}G_D(x,z)\,\mu(dz)<\infty.
\]
From this, Proposition \ref{prop4.wtp1},  and  \cite[Theorem 9.1]{Hansen}
we get the desired result.
\end{proof}

\begin{theorem}
Assume that $a\in (\lambda_1^D,\lambda_1^{D_0})$. Then there exists at most one integral solution to {\rm{(\ref{eq1.1})}}.
\end{theorem}
\begin{proof}
Let $u_1,u_2$ be two integral solutions to (\ref{eq1.1}). We divide the proof into two steps.

Step 1. We shall show that without loss of generality we may assume that $u_1\le u_2$.
Assume that whenever we know that  $w, v$ are integral solutions to (\ref{eq1.1}) such that $w\le v$, then $w=v$.
By Proposition \ref{prop3.3},  $u:=u_1\wedge u_2$ is a an integral supersolution to (\ref{eq1.1}). It is clear that for a sufficiently
small $c>0$, $c\varphi_1^D\le \mathbb I_{D\setminus \overline D_0}$. Therefore, since
\[
-(\Delta^{\alpha/2})_{|D}(c\varphi_1^D)=\lambda_1^Dc\varphi_1^D=ac\varphi_1^D-c(a-\lambda_1^D)\varphi_1^D
\]
and $a>\lambda_1^D$, we see that  $c\varphi_1^D$ is a an integral subsolution to (\ref{eq1.1}) (cf. \eqref{eq.smfe1}).  By the definition of  an integral supersolution
to (\ref{eq1.1}), there exist  nonnegative measures  $\mu,\nu\in\mathbb M_0(D)$ such that
\[
u=aR^Du+R^D\mu-R^D\nu,
\]
and $\nu$ is the reaction measure for $u$.
 By Proposition \ref{prop4.2ms}, the above equation may be equivalently rewritten as
\[
u=aR^Du+R^D\mu-R^D(u\cdot\nu).
\]
By Lemma \ref{lm2.4.1} there exists a MAF $M$ such that for any $x\in D$,
\begin{equation}
\label{eq2.4.6-fk}
u(X_t)=\int_t^{\tau_D}au(X_r)\,dr+\int_t^{\tau_D}\,dA^\mu_r-\int_t^{\tau_D}u(X_r)\,dA^\nu_r-\int_t^{\tau_D}\,dM_r,\quad t\in [0,\tau_D],\quad P_x\mbox{-a.s.}
\end{equation}
By the integration by parts formula applied to the product $e^{-A_t}u(X_t)$, we get
\begin{equation}
\label{eq2.4.6-fk}
e^{-A^\nu_t}u(X_t)=\int_t^{\tau_D}ae^{-A^\nu_r}u(X_r)\,dr+\int_t^{\tau_D}e^{-A^\nu_r}\,dA^\mu_r-\int_t^{\tau_D}e^{-A^\nu_r}\,dM_r,\quad t\in [0,\tau_D],
\end{equation}
$P_x\mbox{-a.s.}$ Taking expectation $E_x$ of both sides of the above equation with $t=0$ yields
\[
u(x)=E_x\int_0^{\tau_D}ae^{-A^\nu_r}u(X_r)\,dr+E_x\int_0^{\tau_D}e^{-A^\nu_r}\,dA^\mu_r,\quad x\in D.
\]
Therefore, by \eqref{ab-fk_1},
\[
u(x)=a\int_DG^\nu_D(x,y) u(y)\,dy+\int_DG^\nu_D(x,y)\,\mu(dy),\quad x\in D.
\]
By Proposition \ref{prop4.1}, Proposition \ref{prop4.2} and Lemma \ref{lm4.1},  $G^\nu_D\sim G_D$.
This when combined with the above equation and (\ref{eq4.1}) gives
\[
u(x)\ge ac\int_D G_D(x,y) u(y)\,dy\ge acc_1\varphi_1^D(x)\int_D\varphi_1^D(y)u(y)\,dy,\quad x\in D,
\]
so $u\ge c\varphi_1^D$ for some $c>0$. Hence, for a sufficiently small $c>0$, $c\varphi_1^D$
is a an integral subsolution to (\ref{eq1.1}) such that $c\varphi_1^D\le u$. By Proposition \ref{prop3.3}, there exists an integral solution $v$ to (\ref{eq1.1})
such that $c\varphi_1^D\le v\le u$. Hence $v\le u_1$ and $v\le u_2$. By the assumption of  Step 1, $u_1=v=u_2$.

Step 2. Assume that $u_1\le u_2$. Let $\nu_1,\nu_2$ be the reaction measures for  $u_1$ and $u_2$, respectively.  Then, by Proposition \ref{prop4.4},
\[
\EE_D(u_1,u_2)+\int_Du_2\,d\nu_1=a(u_1,u_2),\qquad \EE_D(u_2,u_1)+\int_Du_1\,d\nu_2=a(u_2,u_1).
\]
Hence
\[
\int_Du_2\,d\nu_1-\int_Du_1\,d\nu_2=0.
\]
From this and Proposition \ref{prop4.2ms} we conclude that
\begin{equation}
\label{eq.min.ext1}
\int_Du_1u_2(d\nu_1-d\nu_2)=0.
\end{equation}
We can regard $u_i$ as a solution to the following obstacle problem
\[
\max\big\{-(\Delta^{\alpha/2})_{|D} w_i-g_i,w_i- \mathbb I_{D\setminus\overline D_0}\big\}=0,
\]
where $g_i=a u_i,\, i=1,2$. Since $u_1\le u_2$, we have $g_1\le g_2$. Applying  \cite[Proposition 3.12]{K:SM2} yields $d\nu_1\le d\nu_2$. This when combined with (\ref{eq.min.ext1}) and the fact that
$u_1,u_2$ are strictly positive implies that  $\nu_1=\nu_2$. Therefore, we have
\[
(u_2-u_1)=aR^D(u_2-u_1).
\]
Thus,
\begin{equation}
\label{eq4.9}
(u_2-u_1)(x)=a\int_DG_D(x,y)(u_2-u_1)(y)\,dy,\quad x\in D.
\end{equation}
We have assumed that $u_2-u_1\ge 0$. Striving for a contradiction, suppose that $(u_2-u_1)(x)>0$ for some $x\in D$. Then  continuity of $u_1,u_2$ and (\ref{eq4.9})  would imply that
$u_2-u_1$ is strictly positive on $D$, in contradiction with the fact that $a>\lambda_1^D$.
\end{proof}

\begin{remark}
\label{uw.rem1}
All the results of the paper hold for $\alpha=2$. In case $\alpha=2$ the proofs run in the same way as in case $\alpha\in(0,2)$, the  only difference being in the proof of Proposition \ref{prop4.wtp1} and Lemma \ref{lm4.1}. In case $d\ge 3$, Proposition \ref{prop4.wtp1} follows from \cite[Theorem 3.1]{Ri}, and in case $d=2$ it follows from \cite[Corollary 9.6]{Hansen} - we need however $D$ to be finitely connected.
As to the proof of Lemma \ref{lm4.1}, in  case    $d=2$,  to get (\ref{eq.qe.ge}), we use \cite[Theorem 6.24]{CZ},
and  in case $d\ge 3$,  we use \cite[Theorem 6.5]{CZ}.
Instead of \cite[Theorem 1.2]{ChenSong}, we use \cite[Lemma 6.7]{CZ}.


Ultracontractivity of the semigroup generated by  $\Delta_D$ follows from  \cite[Theorem 9.3]{DS}. That Lipschitz bounded domains are Dirichlet regular is well known (see, e.g., \cite[page 350]{BH}).
\end{remark}

\subsection*{Acknowledgements}
{\small This work was supported by Polish National Science Centre
(Grant No. 2017/25/B/ST1/00878).}
\smallskip

{\small Data sharing not applicable to this article as no datasets were generated or analysed during the current study.}

\end{document}